\documentclass{amsart}
\usepackage{amssymb}
\usepackage[all]{xy}

\newtheorem{theorem}{Theorem}[section]
\newtheorem{lemma}[theorem]{Lemma}
\newtheorem{proposition}[theorem]{Proposition}
\newtheorem{corollary}[theorem]{Corollary}

\theoremstyle{definition}

\theoremstyle{remark}
\newtheorem{remark}[theorem]{Remark}

\DeclareMathOperator{\rank}{rank}

\newcommand{\calM}{\mathcal{M}}
\newcommand{\calL}{\mathcal{L}}
\newcommand{\calE}{\mathcal{E}}
\newcommand{\calO}{\mathcal{O}}
\newcommand{\calU}{\mathcal{U}}
\newcommand{\calF}{\mathcal{F}}
\newcommand{\Ldet}{\mathcal{L}_{\det}}
\newcommand{\frakM}{\mathfrak{M}}
\newcommand{\frakU}{\mathfrak{U}}
\newcommand{\bbZ}{\mathbb{Z}}
\newcommand{\bbC}{\mathbb{C}}
\newcommand{\bbP}{\mathbb{P}}
\newcommand{\rmH}{\mathrm{H}}
\newcommand{\rmR}{\mathrm{R}}
\newcommand{\rmT}{\mathrm{T}}
\newcommand{\Pic}{\mathrm{Pic}}
\newcommand{\NS}{\mathrm{NS}}
\newcommand{\Aut}{\mathrm{Aut}}
\newcommand{\End}{\mathrm{End}}
\newcommand{\Hom}{\mathrm{Hom}}
\newcommand{\Ext}{\mathrm{Ext}}
\newcommand{\GL}{\mathrm{GL}}
\newcommand{\PGL}{\mathrm{PGL}}
\newcommand{\Gm}{\mathbb{G}_{\mathrm{m}}}

\newcommand{\semistab}{\mathrm{ss}}
\newcommand{\semisemistab}{\mathrm{(s)s}}
\newcommand{\stab}{\mathrm{s}}
\newcommand{\id}{\mathrm{id}}
\newcommand{\pr}{\mathrm{pr}}
\newcommand{\tr}{\mathrm{tr}}

\newcommand{\longto}[1][]{\stackrel{#1}{\longrightarrow}}

\begin{document}
\bibliographystyle{plain}

\title[Moduli of vector bundles in positive characteristic]{The Picard group of a coarse moduli space of vector bundles in positive characteristic}
\author{Norbert Hoffmann}
\address{Institut f\"ur Mathematik, Fachbereich Mathematik und Informatik, Freie Universit\"at Berlin, Arnimallee 3, 14195 Berlin, Germany}
\email{norbert.hoffmann@fu-berlin.de}

\begin{abstract}
  Let $C$ be a smooth projective curve over an algebraically closed field of arbitrary characteristic.
  Let $\frakM_{r, L}^{\semistab}$ denote the projective coarse moduli scheme of semistable rank $r$ vector bundles over $C$ with fixed determinant $L$.
  We prove $\Pic( \frakM_{r, L}^{\semistab}) = \bbZ$, identify the ample generator, and deduce that $\frakM_{r, L}^{\semistab}$ is locally factorial. In characteristic zero, this has already been proved by Dr\'{e}zet and Narasimhan.
  The main point of the present note is to circumvent the usual problems with Geometric Invariant Theory in positive caracteristic.
\end{abstract}
 
\maketitle

\section{Introduction}

In classification problems for algebro-geometric objects, the Picard group of the moduli space is always a very interesting invariant.
Roughly speaking, it measures how many ways there are to assign to each of the objects in question a one-dimensional vector space, in a suitably functorial way.

In the case of vector bundles with fixed determinant over a smooth projective curve over $\bbC$,
Dr\'{e}zet and Narasimhan proved in their famous paper \cite{drezet-narasimhan} that the Picard group of the coarse moduli scheme is canonically isomorphic to $\bbZ$.
This also yields some information on the singularities of the coarse moduli scheme.

Actually the concept of assigning to each object a one-dimensional vector space is formalised in the notion of a line bundle on the moduli functor, or on the moduli stack. 
However, such a line bundle does not always give a line bundle on the coarse moduli scheme. In characteristic zero, a criterion for when it does is given by Kempf's lemma from Geometric Invariant Theory.
But in positive characteristic, there seems to be no general method to produce all line bundles on a coarse moduli scheme, or more generally on a GIT quotient.

The present note answers this question for the projective coarse moduli scheme $\frakM_{r, L}^{\semistab}$ of rank $r$ vector bundles with fixed determinant $L$ on a smooth projective curve.
We work over an algebraically closed field of arbitrary characteristic, but only the case of positive characteristic is new.

Actually there seems to be a little uncertainty about the definition of $\frakM_{r, L}^{\semistab}$, since taking closed subschemes does not commute with forming GIT quotients in general.
We define the three possible coarse moduli schemes in Section \ref{sec:moduli}, and prove that the canonical morphisms between them are isomorphisms.

Section \ref{sec:Pic} contains the main result that the Picard of $\frakM_{r, L}^{\semistab}$ is canonically isomorphic to $\bbZ$.
We also identify the ample generator and deduce that $\frakM_{r, L}^{\semistab}$ is locally factorial. 
The proofs are based on some recent literature on the Picard group of a corresponding moduli stack, together with a theorem of Faltings that there are enough nonabelian theta functions;
the latter allows us to descend line bundles from the moduli stack to the coarse moduli scheme.

Under some assumptions on $\deg( L)$, the Picard group of $\frakM_{r, L}^{\semistab}$ has also been studied in the preprint \cite{joshi-mehta}, using more advanced tools for positive characteristic.

\subsubsection*{Acknowledgements}
I thank V.B. Mehta for encouraging me to write this note. I also thank the referees for some helpful suggestions.
The work was supported by the SFB 647: Raum - Zeit - Materie.

\section{Moduli of vector bundles with fixed determinant} \label{sec:moduli}

Let $k$ be an algebraically closed field of arbitrary characteristic. Let $C$ be a geometrically irreducible smooth projective curve over $k$ of genus $g \geq 2$. Let
\begin{equation*}
  \calM_{r, d} \supseteq \calM_{r, d}^{\semistab} \xrightarrow{\pi_{r, d}} \frakM_{r, d}^{\semistab}
\end{equation*}
denote the moduli stack of vector bundles $E$ of rank $r$ over $C$ with $\deg( E) = d \in \bbZ$, its open substack where $E$ is semistable, and the corresponding coarse moduli scheme, respectively.
$\calM_{r, d}$ is a smooth irreducible Artin stack over $k$, and $\frakM_{r, d}^{\semistab}$ is a normal irreducible projective variety over $k$ \cite[Th\'{e}or\`{e}me 17]{seshadri}
constructed using Geometric Invariant Theory \cite{GIT}. The morphism $\pi_{r, d}$ is universal in the sense that every morphism from $\calM_{r, d}^{\semistab}$ to a scheme factors uniquely through it.

There are several ways of fixing the determinant $\det( E) := \Lambda^r E$.
Let $\Pic^d( C)$ denote the Picard variety of line bundles $L$ of degree $d$ over $C$, and choose one such line bundle $L$. One may consider the scheme-theoretic fiber
\begin{equation*}
  \frakM_{r, L}^{\semistab} \hookrightarrow \frakM_{r, d}^{\semistab}
\end{equation*}
of the morphism $\det \colon \frakM_{r, d}^{\semistab} \to \Pic^d( C)$ over the point $L$. One also has the stacks
\begin{equation*}
  \calM_{r, = L} \twoheadrightarrow \calM_{r, \cong L} \hookrightarrow \calM_{r, d}
\end{equation*}
where the closed substack $\calM_{r, \cong L}$ is the fiber of $\det \colon \calM_{r, d} \to \Pic^d( C)$ over the point $L$, and $\calM_{r, = L}$
is the $\Gm$-torsor over $\calM_{r, \cong L}$ whose fiber over any point $E$ is the space of all isomorphisms $\phi \colon L \to \det( E)$.
So $\calM_{r, \cong L}$ is the moduli stack of all $E$ such that $\det( E) \cong L$, whereas $\calM_{r, = L}$ is the moduli stack of all pairs $( E, \phi)$ containing an isomorphism $\phi \colon L \to \det( E)$.
Since the trace map
\begin{equation} \label{eq:trace}
  \tr \colon \Ext^1( E, E) \longto \rmH^1( C, \calO_C)
\end{equation}
is surjective in any characteristic, $\calM_{r, \cong L}$ is still smooth; the same follows then for $\calM_{r, = L}$.
Using \cite[Section 1.5]{GIT}, the GIT construction of $\frakM_{r, d}^{\semistab}$ carries over to these fixed determinant situations and provides coarse moduli schemes
\begin{equation*}
  \calM_{r,  = L} \supseteq \calM_{r,  = L}^{\semistab} \xrightarrow{\pi_{r,  = L}} \frakM_{r,  = L}^{\semistab} \qquad\text{and}\qquad
  \calM_{r, \cong L} \supseteq \calM_{r, \cong L}^{\semistab} \xrightarrow{\pi_{r, \cong L}} \frakM_{r, \cong L}^{\semistab}
\end{equation*}
for the open substacks $\calM_{r,  = L}^{\semistab}$ and $\calM_{r, \cong L}^{\semistab}$ where $E$ is semistable. 
The morphisms $\pi_{r,  = L}$ and $\pi_{r, \cong L}$ are again universal among morphisms to schemes.

The problem of comparing the three coarse moduli schemes $\frakM_{r,  = L}^{\semistab}$, $\frakM_{r, \cong L}^{\semistab}$ and $\frakM_{r, L}^{\semistab}$ might seem trivial at first sight,
but it involves slightly delicate issues of GIT in arbitrary characteristic.
\begin{proposition} \label{prop:isos}
  The canonical morphisms
  \begin{equation*}
    \frakM_{r,  = L}^{\semistab} \longto \frakM_{r, \cong L}^{\semistab} \longto \frakM_{r, L}^{\semistab}
  \end{equation*}
  given by the universal properties of $\pi_{r,  = L}$ and $\pi_{r, \cong L}$ are isomorphisms.
\end{proposition}
\begin{proof}
  Choose an ample line bundle $\calO( 1)$ over $X$, and an integer $m \gg 0$ such that $E( m) := E \otimes \calO( 1)^{\otimes m}$ is globally generated with $\rmH^1( C, E( m)) = 0$
  for every semistable $E$ of rank $r$ and degree $d$; then $\rmH^0( C, E( m))$ has the same dimension $N = r \chi( \calO( m)) + d$ for all such $E$. We have the standard presentation
  \begin{equation*}
    \calM^{\semistab}_{r, d} = [ Q^{\semistab}_{r, d} / \GL_N]
  \end{equation*}
  where $Q^{\semistab}_{r, d}$ is the fine moduli scheme of isomorphism classes of pairs $(E, B)$ consisting of a semistable vector bundle $E$ of rank $r$ and degree $d$ together with a basis $B$ of $\rmH^0( C, E( m))$.
  The variety $Q^{\semistab}_{r, d}$ is smooth, and $\GL_N$ acts on it by changing the basis $B$.
  The above construction of the stacks $\calM^{\semistab}_{r, \cong L}$ and $\calM^{\semistab}_{r, = L}$ directly implies that this presentation of $\calM^{\semistab}_{r, d}$ induces presentations
  \begin{equation} \label{eq:presentations}
    \calM^{\semistab}_{r, \cong L} = [ Q^{\semistab}_{r, \cong L} / \GL_N] \qquad\text{and}\qquad \calM^{\semistab}_{r, = L} = [ Q^{\semistab}_{r, = L} / \GL_N]
  \end{equation}
  where $Q^{\semistab}_{r, \cong L}$ is the fiber of $\det \colon Q^{\semistab}_{r, d} \to \Pic^d( C)$ over the point $L$,
  and $Q^{\semistab}_{r, = L}$ is the $\Gm$-torsor over $Q^{\semistab}_{r, \cong L}$ whose fibers parametrize isomorphisms $\phi \colon L \to \det( E)$.

  The morphism $\det \colon Q^{\semistab}_{r, d} \to \Pic^d( C)$ is a submersion, since its differential at any point $( E, B)$ is the composition of the natural surjective linear map
  \begin{equation*}
    \rmT_{(E, B)} Q^{\semistab}_{r, d} \longto \Ext^1( E, E)
  \end{equation*}
  that sends each infinitesimal deformation of a pair $( E, B)$ to the underlying infinitesimal deformation of $E$, followed by the trace map \eqref{eq:trace}, which is surjective as well.
  This shows that $Q^{\semistab}_{r, \cong L}$ and $Q^{\semistab}_{r, = L}$ are also smooth.

  The coarse moduli spaces in question are constructed via GIT as good quotients
  \begin{equation*}
    \frakM_{r, d}^{\semistab} = Q^{\semistab}_{r, d} /\!\!/ \GL_N, \quad \frakM_{r, \cong L}^{\semistab} = Q^{\semistab}_{r, \cong L} /\!\!/ \GL_N \quad\text{and}\quad \frakM_{r, = L}^{\semistab} = Q^{\semistab}_{r, = L} /\!\!/ \GL_N
  \end{equation*}
  in the sense of \cite[Definition 4.2.2]{huybrechts-lehn}. In particular, these are categorical quotients.
  Slightly abusing notation, we denote the quotient morphisms again by
  \begin{equation} \label{eq:quotients}
    Q^{\semistab}_{r, d} \xrightarrow{\pi_{r, d}} \frakM_{r, d}^{\semistab}, \quad Q^{\semistab}_{r, \cong L} \xrightarrow{\pi_{r, \cong L}} \frakM_{r, \cong L}^{\semistab}
      \quad\text{and}\quad Q^{\semistab}_{r, = L} \xrightarrow{\pi_{r, = L}} \frakM_{r, = L}^{\semistab}.
  \end{equation}

  The center $\Gm \subseteq \GL_N$ acts trivially on $Q^{\semistab}_{r, d}$ and $Q^{\semistab}_{r, \cong L}$, but not on $Q^{\semistab}_{r, = L}$.
  More precisely, $\lambda \cdot \id_E$ provides an isomorphism between the pairs $( E, B)$ and $( E, \lambda \cdot B)$ for any nonzero scalar $\lambda$,
  but only between the triples $( E, B, \phi)$ and $( E, \lambda \cdot B, \lambda^r \cdot \phi)$.

  Thus we see that the (possibly non-reduced) subgroup $\mu_r \subset \Gm$ acts trivially on $Q^{\semistab}_{r, = L}$,
  and the (scheme-theoretic) factor group $\Gm/\mu_r \cong \Gm$ acts freely on $Q^{\semistab}_{r, = L}$ with quotient $Q^{\semistab}_{r, \cong L}$.
  In particular, every $\GL_N$-invariant morphism from $Q^{\semistab}_{r, = L}$ to a scheme factors uniquely through $Q^{\semistab}_{r, \cong L}$.
  This means that the canonical map between the categorical quotients $\frakM_{r,  = L}^{\semistab}$ and $\frakM_{r, \cong L}^{\semistab}$ is an isomorphism.

  Now we compare $\frakM_{r, \cong L}^{\semistab}$ and $\frakM_{r, L}^{\semistab}$.
  The former is a GIT quotient of a closed subscheme of $Q^{\semistab}_{r, d}$, whereas the latter is a closed subscheme of a GIT quotient of $Q^{\semistab}_{r, d}$.
  These two operations commute in characteristic $0$, but they do not commute in general in positive characteristic; this is the main technical issue here.

  Sending each line bundle $\xi$ of degree $0$ over $C$ to $L \otimes \xi^{\otimes r}$ defines a morphism
  \begin{equation*}
    J( C) := \Pic^0( C) \xrightarrow{\tau_{r, L}} \Pic^d( C).
  \end{equation*}
  We consider the pullback diagram
  \begin{equation*} \xymatrix{
    Q_{r, d}^{\semistab} \times_{\Pic^d( C)} J( C) \ar[r] \ar[d]_{\pi_{r, d} \times \id} & Q_{r, d}^{\semistab} \ar[d]^{\pi_{r, d}}\\
    \frakM_{r, d}^{\semistab} \times_{\Pic^d( C)} J( C)  \ar[r] \ar[d] & \frakM_{r, d}^{\semistab} \ar[d]^{\det}\\
    J( C) \ar[r]^{\tau_{r, L}} & \Pic^d( C).
  } \end{equation*}
  Since $\pi_{r, d}$ is a uniform categorical quotient by \cite[Theorem A.1.1]{GIT}, and $\tau_{r, L}$ is flat, the pullback $\pi_{r, d} \times \id$ is also a categorical quotient modulo the action of $\PGL_N$.

  Sending each pair $(E, B)$ and each line bundle $\xi$ to $E \otimes \xi^{-1}$ defines a morphism
  \begin{equation*}
    Q_{r, d}^{\semistab} \times J( C) \longto \frakM_{r, d}^{\semistab}.
  \end{equation*}
  It restricts, by definition of $\tau_{r, L}$, to a morphism
  \begin{equation*}
    Q_{r, d}^{\semistab} \times_{\Pic^d( C)} J( C) \longto \frakM_{r, \cong L}^{\semistab}.
  \end{equation*}
  This morphism is clearly $\PGL_N$-invariant, and hence descends to a morphism
  \begin{equation*}
    \frakM_{r, d}^{\semistab} \times_{\Pic^d( C)} J( C) \longto \frakM_{r, \cong L}^{\semistab}.
  \end{equation*}
  Its restriction to the (scheme-theoretic) fiber over the origin in $J( C)$ is a morphism
  \begin{equation*}
    \frakM_{r, L}^{\semistab} \longto \frakM_{r, \cong L}^{\semistab}.
  \end{equation*}
  It is straightforward to check that this morphism is a two-sided inverse of the canonical morphism $\frakM_{r, \cong L}^{\semistab} \to \frakM_{r, L}^{\semistab}$ in question.
\end{proof}
In order to simplify the notation, we will from now on identify $\frakM_{r,  = L}^{\semistab}$ and $\frakM_{r, \cong L}^{\semistab}$ with $\frakM_{r, L}^{\semistab}$ via the canonical isomorphisms in Proposition \ref{prop:isos}.
\begin{corollary} \label{cor:O}
  The canonical homomorphisms
  \begin{equation*}
    \calO_{\frakM_{r, L}^{\semistab}} \longto (\pi_{r, \cong L})_*( \calO_{\calM_{r, \cong L}^{\semistab}}) \quad\text{and}\quad \calO_{\frakM_{r, L}^{\semistab}} \longto (\pi_{r, = L})_*( \calO_{\calM_{r, = L}^{\semistab}})
  \end{equation*}
  are isomorphisms of Zariski sheaves.
\end{corollary}
\begin{proof}
  We have seen in the previous proof that $\frakM_{r, L}^{\semistab}$ is a good quotient of both $Q_{r, \cong L}^{\semistab}$ and $Q_{r, = L}^{\semistab}$ modulo $\GL_N$. In particular, the canonical homomorphisms
  \begin{equation*}
    \calO_{\frakM_{r, L}^{\semistab}} \longto (\pi_{r, \cong L})_*( \calO_{Q_{r, \cong L}^{\semistab}})^{\GL_N} \quad\text{and}\quad \calO_{\frakM_{r, L}^{\semistab}} \longto (\pi_{r, = L})_*( \calO_{Q_{r, = L}^{\semistab}})^{\GL_N}
  \end{equation*}
  induced by the quotient morphisms in \eqref{eq:quotients} are isomorphisms.
  Due to the presentations in \eqref{eq:presentations}, these invariant direct images are precisely the direct images of the structure sheaves of the stacks in question.
\end{proof}

\section{The Picard group of the coarse moduli scheme} \label{sec:Pic}
We keep the notations of the previous section.
In particular, $L$ is a line bundle of degree $d$ over the curve $C$, and $\calM_{r, = L}$ is the moduli stack of rank $r$ vector bundles $E$ over $C$ together with an isomorphisms $L \to \det( E)$.
We assume $r \geq 2$.

Let $\Ldet$ denote the determinant of cohomology line bundle over $\calM_{r, d}$. Its fiber over the moduli point of a vector bundle $E$ is by definition $\det \rmH^0( E) \otimes \det^{-1} \rmH^1( E)$.
To describe this line bundle more precisely, let $\calE$ be a vector bundle over $C \times S$ for some $k$-scheme $S$.
Then the complex $\rmR \pr_{2, *}( \calE)$ over $S$ is perfect by \cite{EGA3}, so after replacing $S$ by an open covering, we may assume
\begin{equation*}
  \rmR \pr_{2, *}( \calE) \cong [ \calF^0 \longto[ \partial] \calF^1 ]
\end{equation*}
with vector bundles $\calF^0$ and $\calF^1$ over $S$. In this case, the pullback of $\Ldet$ to $S$ is
\begin{equation*}
  \det \rmR \pr_{2, *}( \calE) \cong \det( \calF^0) \otimes \det( \calF^1)^*;
\end{equation*}
see \cite{knudsen-mumford} for more details on the construction of such determinant line bundles.

In the special case $d = r( g-1)$, vector bundles $E$ of rank $r$ and degree $d$ satisfy $\chi( E) = 0$ by Riemann-Roch.
Then $\rank( \calF^0) = \rank( \calF^1)$, so $\det( \partial)$ is a section of the line bundle $\det( \calF^0)^* \otimes \det( \calF^1)$;
these sections patch together to a canonical section
\begin{equation} \label{eq:theta}
  \vartheta \in \rmH^0( \calM_{r, r( g-1)}, \Ldet^*)
\end{equation}
which is known as a (nonabelian) theta function.

Slightly abusing notation, we will denote the pullback of $\Ldet$ to $\calM_{r, =L}$ again by $\Ldet$, and also its restriction to the open substacks
\begin{equation*}
  \calM_{r, = L} \supseteq \calM_{r, = L}^{\semistab} \supseteq \calM_{r, = L}^{\stab}
\end{equation*}
defined by the condition that $E$ is semistable or stable, respectively.
\begin{theorem}
  The group $\Pic( \calM_{r, = L})$ is freely generated by $\Ldet$.
\end{theorem}
\begin{proof}
  The case $k = \bbC$ is contained in \cite[Remark 7.11 and Proposition 9.2]{beauville-laszlo}. For arbitrary characteristic and $L$ trivial, it is proved in \cite[Theorem 17]{faltings_2003}.
  The generalization to nontrivial line bundles $L$ is carried out in \cite[Proposition 4.2.3]{biswas-hoffmann_2010}.

  More precisely, the latter reference shows $\Pic( \calM_{r, = L}) \cong \bbZ$. In order to determine the image of $\Ldet$ under this isomorphism, we use the morphism
  \begin{equation*}
    \varphi \colon \calM_{1, 0} \longto \calM_{r, = L}
  \end{equation*}
  defined by sending each line bundle $\xi$ of degree $0$ to $( L \otimes \xi) \oplus \calO_C^{r-2} \oplus \xi^{-1}$. We also use the canonical homomorphism of abelian groups
  \begin{equation*}
    c: \Pic( \calM_{1, 0}) \longto \End J( C)
  \end{equation*}
  given by \cite[Section 3.2]{biswas-hoffmann_2010}; in the notation of that section, it is the composition
  \begin{equation*}
    \Pic( \calM_{\Gm}^0) \xrightarrow{c_{\Gm}} \NS( \calM_{\Gm}) \xrightarrow{\pr_2} \Hom^s( \bbZ \otimes \bbZ, \End J_C) \hookrightarrow \End J_C.
  \end{equation*}
  Using the standard isomorphism
  \begin{equation*}
    \det \rmR \pr_{2, *}( \calE_1 \oplus \calE_2) \cong \det \rmR \pr_{2, *}( \calE_1) \otimes \det \rmR \pr_{2, *}( \calE_2)
  \end{equation*}
  together with \cite[Lemma 4.4.1 and Remark 3.2.3]{biswas-hoffmann_2010}, we see that
  \begin{equation*}
    c \circ \varphi^* \colon \Pic( \calM_{r, = L}) \longto \End J( C) 
  \end{equation*}
  maps $\Ldet$ to $-2 \cdot \id_{J( C)}$.
  On the other hand, $c \circ \varphi^*$ maps the two generators of $\Pic( \calM_{r, = L})$ to $\pm 2 \cdot \id_{J( C)}$ according to \cite[Proposition 4.4.7 and Remark 3.2.3]{biswas-hoffmann_2010}.
  This shows that $\Ldet$ generates $\Pic( \calM_{r, = L})$.
\end{proof}
\begin{corollary} \label{cor:Pic_stack}
  i) The group $\Pic( \calM_{r, = L}^{\semistab})$ is freely generated by $\Ldet$.

  ii) The group $\Pic( \calM_{r, = L}^{\stab})$ is generated by $\Ldet$.
\end{corollary}
\begin{proof}
  Since $\calM_{r, = L}$ is a smooth Artin stack, the restriction maps
  \begin{equation*}
    \Pic( \calM_{r, = L}) \longto \Pic( \calM_{r, = L}^{\semistab}) \longto \Pic( \calM_{r, = L}^{\stab})
  \end{equation*}
  are surjective; see for example \cite[Lemma 7.3]{biswas-hoffmann_2012}. The first of these maps is also injective, because the complement of $\calM_{r, = L}^{\semistab}$ in $\calM_{r, = L}$ has codimension $\geq 2$.
  The latter follows from the fact that $\calM_{r, = L}$ is smooth of dimension $(g-1)(r^2 - 1)$, whereas the moduli stack of triples $(E, \phi, E_1)$ with $( E, \phi)$ in $\calM_{r, = L}$ and $E_1 \subseteq E$ a subbundle
  of fixed rank $0 < r_1 < r$ and degree $d_1$ is smooth of dimension
  \begin{equation*}
    (g-1)(r_1^2 + (r - r_1)^2 - 1) + r_1(r - r_1) \big( g - 1 - \frac{d_1}{r_1} + \frac{d - d_1}{r - r_1} \big)
  \end{equation*}
  according to \cite[Proposition A.3]{hoffmann}; for $d_1/r_1 > d/r$, this is $\leq (g-1)(r^2 - 1) - 2$.
\end{proof}
The stable locus $\calM_{r, = L}^{\stab} \subseteq \calM_{r, = L}^{\semistab}$ is the inverse image of an open subscheme
\begin{equation*}
  \frakM_{r, L}^{\stab} \subseteq \frakM_{r, L}^{\semistab}.
\end{equation*}
\begin{lemma} \label{lemma:injective}
  The canonical group homomorphisms
  \begin{equation*}
    \Pic( \frakM_{r, L}^{\semistab}) \longto \Pic( \calM_{r, = L}^{\semistab}) \quad\text{and}\quad \Pic( \frakM_{r, L}^{\stab}) \longto \Pic( \calM_{r, = L}^{\stab})
  \end{equation*}
  given by pullback along the morphism $\pi_{r, = L}$ are injective.
\end{lemma}
\begin{proof}
  Let $\calL$ be a line bundle on $\frakM_{r, L}^{\semisemistab}$ such that $(\pi_{r, = L})^*( \calL)$ is trivial on $\calM_{r, = L}^{\semisemistab}$.
  Then Corollary \ref{cor:O} implies that $(\pi_{r, = L})_* (\pi_{r, = L})^*( \calL)$ is also a trivial line bundle.
  Using the projection formula and again Corollary \ref{cor:O}, it follows that $\calL$ is trivial.
\end{proof}
\begin{corollary}
  The group $\Pic( \frakM_{r, L}^{\semistab})$ is isomorphic to $\bbZ$.
\end{corollary}
\begin{remark}
  In the case $L = \calO_C$ of vector bundles with trivial determinant, similar arguments have been given in \cite[Theorem 7]{osserman}.
\end{remark}
\begin{lemma} \label{lemma:inertia}
  The line bundle $\Ldet^{\otimes n}$ on $\calM_{r, = L}^{\stab}$ can only be isomorphic to the pullback of a line bundle on $\frakM_{r, L}^{\stab}$ if $n$ is a multiple of $r/\gcd( r, d)$.
\end{lemma}
\begin{proof}
  Since $\Ldet$ is a line bundle on the stack $\calM_{r, d}$, the automorphism group $\Aut( E)$ of the vector bundle $E$ acts on the fiber of $\Ldet$ over the moduli point of $E$ in $\calM_{r, d}$.
  The subgroup $\Gm \subseteq \Aut( E)$ of scalar automorphisms acts on this one-dimensional vector space with weight $\chi := r( 1 - g) + d$ according to Riemann-Roch.

  Similarly, the automorphism group $\Aut( E, \phi)$ of the pair $( E, \phi \colon L \longto[\sim] \det( E))$ acts on the fiber of $\Ldet$ over the moduli point of $(E, \phi)$ in $\calM_{r, = L}$.
  The (possibly non-reduced) subgroup scheme $\mu_r \subseteq \Aut( E, \phi)$ of scalar automorphisms still acts with weight $\chi$.
  If an integer $n$ is not divisible by $r/\gcd(r, d)$, then $n \cdot \chi$ is not divisible by $r$, so this action of $\mu_r$ on the fibers of $\Ldet^{\otimes n}$ is nontrivial.

  On the other hand, all automorphism group schemes of $\calM_{r, = L}^{\stab}$ act trivially on the fibers of any line bundle that is pulled back from $\frakM_{r, L}^{\stab}$.
\end{proof}
\begin{theorem}
  i) The line bundle $\Ldet^{\otimes -r/\gcd(r, d)}$ on $\calM_{r, = L}^{\semistab}$ is isomorphic to the pullback of an ample line bundle $\calL( \Theta)$ on $\frakM_{r, L}^{\semistab}$.

  ii) $\calL( \Theta)$ generates $\Pic( \frakM_{r, L}^{\semistab}) \cong \bbZ$, and its restriction generates $\Pic( \frakM_{r, L}^{\stab})$.
\end{theorem}
\begin{proof}
  Let $F$ be a vector bundle of some rank $n \geq 1$ over $C$ such that
  \begin{equation*}
    \deg( F)/n + d/r = g-1.
  \end{equation*}
  Sending each pair $(E, \phi)$ to the vector bundle $E \otimes F$ defines a morphism of stacks
  \begin{equation*}
    \calM_{r, = L}^{\semistab} \longto \calM_{nr, nr(g-1)}.
  \end{equation*}
  It is easy to check that the pullback of $\Ldet$ along this morphism is isomorphic to $\Ldet^{\otimes n}$. Let us denote the pullback of the canonical section $\vartheta$ in \eqref{eq:theta} by
  \begin{equation*}
    \vartheta_F \in \rmH^0( \calM_{r, = L}^{\semistab}, \Ldet^{\otimes -n}).
  \end{equation*}
  The locus where $\vartheta_F$ does not vanish is by construction the open substack
  \begin{equation*}
    \calU_F \subseteq \calM_{r, = L}^{\semistab}
  \end{equation*}
  defined by the condition $\rmH^0( C, E \otimes F) = \rmH^1( C, E \otimes F) = 0$. S-equivalence preserves this condition, so $\calU_F$ is the inverse image of an open subscheme
  \begin{equation*}
    \frakU_F \subseteq \frakM_{r, L}^{\semistab}
  \end{equation*}
  since $\frakM_{r, L}^{\semistab}$ is a good quotient. The section $\vartheta_F$ trivializes the line bundle $\Ldet^{\otimes -n}$ over $\calU_F$.
  Using Corollary \ref{cor:O}, it follows that the sheaf $(\pi_{r, = L})_*( \Ldet^{\otimes -n})$ is isomorphic to the structure sheaf over $\frakU_F$, and that the canonical homomorphism
  \begin{equation*}
    (\pi_{r, = L})^* (\pi_{r, = L})_*( \Ldet^{\otimes -n}) \longto \Ldet^{\otimes -n}
  \end{equation*}
  is an isomorphism over $\calU_F$. According to \cite[Lemma 3.1 and Remark 3.2]{seshadri_1993}, we have
  \begin{equation*}
    \frakM_{r, L}^{\semistab} = \bigcup_{\rank( F) = n} \frakU_F
  \end{equation*}
  for every sufficiently large multiple $n$ of $r/\gcd(r, d)$; this was first shown by Faltings \cite{faltings_1993}.
  Under this assumption on $n$, it follows that $(\pi_{r, = L})_*( \Ldet^{\otimes -n})$ is a line bundle on $\frakM_{r, L}^{\semistab}$ which pulls back to $\Ldet^{\otimes -n}$ on $\calM_{r, = L}^{\semistab}$.
  Taking the difference of these line bundles for two successive values of $n$, we get a line bundle $\calL( \Theta)$ on $\frakM_{r, L}^{\semistab}$ with
  \begin{equation*}
     (\pi_{r, = L})^* \calL( \Theta) \cong \Ldet^{\otimes -r/\gcd( r, d)}
  \end{equation*}
  on $\calM_{r, = L}^{\semistab}$.
  This line bundle and its restriction generate $\Pic( \frakM_{r, L}^{\semistab})$ and $\Pic( \frakM_{r, L}^{\stab})$ due to Corollary \ref{cor:Pic_stack}, Lemma \ref{lemma:injective} and Lemma \ref{lemma:inertia}.
  In particular, $\calL( \Theta)$ or its dual is ample on $\Pic( \frakM_{r, L}^{\semistab})$.
  But we have also seen that every sufficiently large power of $\calL( \Theta)$ is globally generated; hence $\calL( \Theta)$, and not its dual, is ample.
\end{proof}
\begin{corollary}
  The projective variety $\frakM_{r, L}^{\semistab}$ is locally factorial.
\end{corollary}
\begin{proof}
  In the case $g = r = 2$ and $d = \deg( L)$ even, the moduli space $\frakM_{r, L}^{\semistab}$ is isomorphic to $\bbP^3$ by \cite{narasimhan-ramanan, bhosle}, and hence locally factorial.

  In all other cases, the complement of $\frakM_{r, L}^{\stab}$ has codimension $\geq 2$ in $\frakM_{r, L}^{\semistab}$.
  In particular, the Picard group of $\frakM_{r, L}^{\stab}$ coincides with the Picard group of the smooth locus of $\frakM_{r, L}^{\semistab}$.
  Every line bundle on this smooth locus can be extended to $\frakM_{r, L}^{\semistab}$ due to the previous theorem.
  This implies that $\frakM_{r, L}^{\semistab}$ is locally factorial.
\end{proof}

\begin{remark}
  The singularities of $\frakM_{r, L}^{\semistab}$ have been studied in \cite{balaji-mehta}. In particular, it is proved there that $\frakM_{r, L}^{\semistab}$ is Gorenstein.
\end{remark}

\bibliography{picardgroup}

\end{document}